\documentclass[a4paper,10pt]{amsart}

\usepackage{amsmath}
\usepackage{amsthm}
\usepackage{amssymb}
\usepackage{amsfonts}
\usepackage{mathrsfs}  
\usepackage{enumitem} 
\usepackage{xcolor}
\usepackage[bookmarks=false,hyperindex,pdftex,colorlinks,citecolor=blue,urlcolor=cyan]{hyperref}
\usepackage[a4paper,lmargin=3cm,rmargin=3cm,tmargin=4cm,bmargin=4cm,marginparwidth=2.8cm,marginparsep=1mm]{geometry} 
\usepackage{soul}

\DeclareMathOperator{\lspan}{span}                          
\DeclareMathOperator{\conv}{conv}                           
\DeclareMathOperator{\supp}{supp}                           
\DeclareMathOperator{\diam}{diam}                           
\DeclareMathOperator{\Lip}{Lip}                             

\newcommand{\NN}{\mathbb{N}}                                
\newcommand{\ZZ}{\mathbb{Z}}                                
\newcommand{\RR}{\mathbb{R}}                                
\newcommand{\ep}{\varepsilon}

\newcommand{\abs}[1]{\left|{#1}\right|}                     
\newcommand{\pare}[1]{\left({#1}\right)}                    
\newcommand{\set}[1]{\left\{{#1}\right\}}                   
\newcommand{\norm}[1]{\left\|{#1}\right\|}                  
\newcommand{\duality}[1]{\left<{#1}\right>}                 
\newcommand{\cl}[1]{\overline{#1}}                          
\newcommand{\restrict}{\mathord{\upharpoonright}}           
\newcommand{\compl}{\stackrel{c}{\hookrightarrow}}          

\newcommand{\lipfree}[1]{\mathcal{F}({#1})}                 
\newcommand{\lipnorm}[1]{\norm{#1}_L}                       

\renewcommand{\leq}{\leqslant}
\renewcommand{\geq}{\geqslant}

\theoremstyle{plain}
\newtheorem{theorem}{Theorem}[section]
\newtheorem{lemma}[theorem]{Lemma}

\newtheorem{proposition}[theorem]{Proposition}

\newtheorem*{claim*}{Claim}
\newtheorem{fact}[theorem]{Fact}
\newtheorem{question}{Question}

\theoremstyle{definition}
\newtheorem*{definition*}{Definition}
\newtheorem{definition}[theorem]{Definition}

\theoremstyle{remark}
\newtheorem{remark}[theorem]{Remark}

\begin{document}
\title{Lipschitz extension and Lipschitz-free spaces over nets in normed spaces}

\author[R. J. Aliaga]{Ram\'on J. Aliaga}
\address[R. J. Aliaga]{Instituto Universitario de Matem\'atica Pura y Aplicada,
Universitat Polit\`ecnica de Val\`encia,
Camino de Vera S/N,
46022 Valencia, Spain}
\email{ramon.aliaga@upv.es}

\author[R. Medina]{Rub\'en Medina}
\address[R. Medina]{Universidad P\'ublica de Navarra (ETSIIIT). Departamento de Estad\'istica, Inform\'atica y Matem\'aticas. Institute for Advanced Materials and Mathematics. 31009-Navarra (Spain). \newline
\href{https://orcid.org/0000-0002-4925-0057}{ORCID: \texttt{0000-0002-4925-0057}}}
\email{ruben.medina@unavarra.es}

\begin{abstract}
We consider subsets $S$ of a metric space $M$ such that Lipschitz mappings defined on $S$ can be extended to Lipschitz mappings on $M$,
and we show that the union of such subsets has the same property under appropriate geometric conditions.
We then derive several consequences to the isomorphic structure and classification of Lipschitz and Lipschitz-free spaces.
Our main result is that the Lipschitz-free space $\mathcal{F}(M)$ is isomorphic to its countable $\ell_1$-sum when $M$ is either a net $N_X$ in any Banach space $X$ or the integer grid $\mathbb{Z}_{\ell_1}$ in $\ell_1$. We also prove that the Lipschitz space $\mathrm{Lip}_0(\mathbb{Z}_{\ell_1})$ is isomorphic to $\mathrm{Lip}_0(\ell_1)$ and that $\mathrm{Lip}_0(N_X)$ contains a complemented copy of $\mathrm{Lip}_0(X)$, among other results.
This answers questions raised by Albiac, Ansorena, C\'uth and Doucha and Candido, C\'uth and Doucha, respectively, and extends previous results by the same authors as well as H\'ajek and Novotn\'y. 
\end{abstract}



\maketitle

\section{Introduction}

\subsection{Motivation and background}

The Lipschitz extension problem has received substantial attention from researchers in Metric Geometry and Computer Science in recent years (see, for instance, \cite{BrudnyiBrudnyi07, LN5, Naor24} and references therein). Specifically, given a metric space $M$, a subset $S$ of $M$ and a real Banach space $X$, one wonders about the existence of a constant $L\geq1$ such that every 1-Lipschitz mapping $f:S\to X$ can be extended to an $L$-Lipschitz mapping $F:M\to X$. Of particular interest is the case where $M$ and $S$ are fixed and one looks for the least extension constant $L$ that works for any range space $X$. This constant is represented by $e(S,M)$\footnote{We warn the reader that this notation is not universal. For instance, other works such as \cite{LN5} use the notation $e(M,X)$ for the optimal extension constant when $X$ is fixed and $S$ runs through the subsets of $M$. Our convention is more consistent with the notation $\lambda(S,M)$ used e.g. in \cite{BrudnyiBrudnyi08}.} and turns out to be the smallest norm of a pointwise continuous linear extension operator $E:\Lip_0(S)\to\Lip_0(M)$ (see Definition \ref{def:lip extendable} and Fact \ref{clipfact}). Here, $\Lip_0(M)$ denotes the Banach space of Lipschitz functions $M\to\RR$ that vanish at a prescribed point. If the extension operator is not required to be pointwise continuous, its smallest norm is denoted $\lambda(S,M)$ instead (see e.g. \cite{BrudnyiBrudnyi08}).

Our main interest in this work is the isomorphic classification of Lipschitz spaces $\Lip_0(M)$ and their preduals, the Lipschitz-free spaces denoted by $\lipfree{M}$. Many of the existing results and techniques are based on obtaining isomorphisms by means of Pe\l czy\'nski's decomposition method (see for instance \cite{CandidoKaufmann1,CandidoKaufmann2,GP,Gartland25,Kaufmann}), which involves determining complementability relations as well as spaces isomorphic to their countable $\ell_p$-sums (where $p=1$ for Lipschitz-free spaces and $p=\infty$ for Lipschitz spaces).
Lipschitz-free spaces isomorphic to their infinite $\ell_1$-sums and Lipschitz spaces isomorphic to their infinite $\ell_\infty$-sums have been thoroughly studied in the last decade (see, for instance, \cite{AACD3,CCD,CG23,HajekNovotny}). On the other hand, the search for complementability relations is intimately connected to the study of Lipschitz extension constants. Indeed, the existence of a pointwise continuous extension operator $E:\Lip_0(S)\to\Lip_0(M)$ is equivalent to the existence of a projection from $\lipfree{M}$ onto $\lipfree{S}$, with the same norm as $E$ (see Fact \ref{clipfact}).

Many of the recent efforts in the study of extension constants focus on estimating the global constants $$e(M)=\sup\set{e(S,M):S\subset M}$$ and $\lambda(M)$ (defined analogously) for different classes of metric spaces $M$ (see, for instance, \cite{Naor15,Naor24}). However, these constants may be infinite in general. This fact limits their applicability to the study of the structure of Lipschitz(-free) spaces to the cases where $e(M)<\infty$, and the main examples where this holds correspond to spaces $M$ that are ``finite-dimensional'' in some sense, for instance doubling spaces \cite{BDMS,LN5}. In order to overcome this issue, we deal with the full constants $e(S,M)$ and consider the family of subsets $S\subset M$ for which $e(S,M)$ or $\lambda(S,M)$ are finite, which we call \emph{(weak$^*$) Lipschitz extendable subsets}. Our main observation, from which most of our results are derived, is that Lipschitz extendability is sometimes stable under taking unions, provided that some natural geometric restrictions are satisfied.

\subsection{Summary of results}

Section \ref{extensions} is devoted to the Lipschitz extension problem. We review its relation to the structure of Lipschitz and Lipschitz-free spaces and we state and prove our main technical tool, Theorem \ref{th:lip extendable infinite union}. It provides sufficient geometric conditions on an infinite family of Lipschitz extendable subsets of a metric space so that the union is also Lipschitz extendable. In a nutshell, this is possible when the sets are separated enough, not too big, and Lipschitz extendable with a uniform constant.

In Section \ref{applications}, we apply Theorem \ref{th:lip extendable infinite union} and the techniques behind it to establish isomorphism and complementation between certain Lipschitz and Lipschitz-free spaces. Many of these relations provide full or partial solutions to various questions raised in the literature.

First, we consider homogeneous metric spaces. Homogeneity allows us to move sets around in order to fulfill the conditions required to apply Theorem \ref{th:lip extendable infinite union}. Combining this with Kalton's decomposition theorem \cite{Kalton}, we prove that $\lipfree{M}$ is isomorphic to its $\ell_1$-sum if the metric space $M$ is homogeneous and unbounded and its closed balls are uniformly Lipschitz extendable to $M$ (Theorem \ref{th:homogeneous unbounded extendable}).

Next, we consider the case where the metric space is a net $N_X$ in a Banach space $X$. Lipschitz-free spaces over nets are targets of study for researchers in both Metric Geometry and in Functional Analysis, as they are intimately connected to the coarse Lipschitz geometry of the space (see, for instance, \cite{CDW16, HM1, HM2, HajekNovotny}).
In this context, Albiac et al asked in \cite[Question 6.5]{AACD3} whether $\lipfree{N_X}$ is always isomorphic to its $\ell_1$-sum. The dual version of the question (i.e. whether $\Lip_0(N_X)$ is isomorphic to its $\ell_\infty$-sum) was also raised by Candido, C\'uth and Doucha in \cite[Question 4]{CCD}. Partial solutions are known: H\'ajek and Novotn\'y gave a positive answer for nets in certain classes of Banach spaces with Schauder bases in 2017 \cite[Theorem 8]{HajekNovotny}, whereas  Albiac et al solved it for nets in doubling spaces in 2021 \cite[Corollary 5.10]{AACD3}. We give a full affirmative answer in Theorem \ref{th:net l1 sum} by reducing the problem to the homogeneous case.

We also consider the relation between the Lipschitz spaces $\Lip_0(N_X)$ and $\Lip_0(X)$. It is asked in \cite[Question 3]{CCD} whether these spaces are always isomorphic, by analogy with the one-dimensional case where $\Lip_0(\RR)=L_\infty$ and $\Lip_0(\ZZ)=\ell_\infty$. The authors of \cite{CCD} provide a positive answer for finite-dimensional $X$. Here, we use Theorem \ref{th:lip extendable infinite union} to provide a partial solution to the problem by showing that $\Lip_0(X)$ is isomorphic to a complemented subspace of $\Lip_0(N_X)$ (Theorem \ref{th:lip net complemented}). With this knowledge, \cite[Question 3]{CCD} is now equivalent to the following question, i.e. whether the converse of Theorem \ref{th:lip net complemented} holds as well.

\begin{question}\label{complequest}
    Is it true that, for every Banach space $X$ and every net $N_X$ of $X$, the space $\Lip_0(N_X)$ is isomorphic to a complemented subspace of $\Lip_0(X)$?
\end{question}

\noindent A natural way to tackle this question would be to try to find a Lipschitz extension operator $E:\Lip_0(N_X)\to\Lip_0(X)$, but this is known to not always be possible \cite{Naor15}. In fact, if $\dim(X)=n$, Bourgain showed that almost extensions of Lipschitz maps from an $(\varepsilon,\varepsilon)$-net in the unit sphere of $X$ are possible but they incur an error that grows linearly with $n$ \cite{Bou86}. Moreover, Naor showed in \cite{Naor21} that this error bound is almost sharp in certain cases. Hence, finding a complemented copy of $\Lip_0(N_X)$ inside $\Lip_0(X)$ may be a nontrivial task.

The closest we are to a positive solution of \cite[Question 3]{CCD}, or equivalently Question \ref{complequest}, is our Theorem \ref{th:grid ell1}, where we prove that the Lipschitz space $\Lip_0(\ZZ_{\ell_1})$ over the integer grid of $\ell_1$ is isomorphic to $\Lip_0(\ell_1)$. We remark that the grid is not a net in $\ell_1$, so we still do not know of any infinite-dimensional Banach space for which Question \ref{complequest} is answered affirmatively. We also show in Theorem \ref{th:gridsum} that the Lipschitz-free space $\lipfree{\ZZ_{\ell_1}}$ is isomorphic to its countable $\ell_1$-sum by means of a specially tailored version of our main tool Theorem \ref{th:lip extendable infinite union} (see Lemma \ref{lm:lemma_extension_partition_ell1}). This answers a question raised in \cite[p. 2717]{CCD}.

Our next result, Theorem \ref{th:compactfinal}, uses a similar technique to show that if $X$ is a Banach space with a finite-dimensional decomposition then there is a compact, convex and generating subset $K$ of $X$ for which $\Lip_0(X)$ is isomorphic to $\Lip_0(K)$. This line of research initiated in \cite{GP}, where Garc\'ia-Lirola and Proch\'azka asked whether the Lipschitz-free space over a separable Banach space $X$ is always isomorphic to the Lipschitz-free space over some compact metric space, preferably a subset of $X$, and gave one example (Pe\l czy\'nski's universal basis space) for which this holds. It continued with the work of Basset \cite{Basset}, where the author gives the first examples of separable metric spaces $M$ (not Banach spaces) such that $\lipfree{M}$ is not isomorphic to the Lipschitz-free space over any compact space. The Banach space case is still open, and our result provides a partial affirmative answer.

We finish our paper by addressing a question posed by the first named author in \cite[Question 5]{Aliaga25} and showing that, for $n,m\in\NN$, the Lipschitz-free space $\lipfree{\RR^n\times\ZZ^m}$ is isomorphic to $\lipfree{\RR^n}\oplus\lipfree{\ZZ^{n+m}}$ (Proposition \ref{prop:bundles}).

\subsection{Definitions and preliminaries}

We will only consider real scalars throughout this document. For any undefined notions related to Banach spaces or Lipschitz spaces we refer to \cite{Banach} or \cite{Weaver2}, respectively.

The symbol $X$ will be reserved for Banach spaces, and the closed unit ball of $X$ will be denoted by $B_X$. By ``operator'' we mean a bounded linear operator between Banach spaces, and similarly for ``projection''. Given Banach spaces $X,Y$, we will write $X\equiv Y$ if they are linearly isometric, $X\sim Y$ if they are linearly isomorphic, and $X\compl Y$ if $Y$ contains a complemented subspace that is isomorphic to $X$. For $p\in\set{1,\infty}$, the $\ell_p$-sum of a sequence $(X_n)$ of Banach spaces will be denoted by $\pare{\bigoplus_{n\in\NN}X_n}_p$. If all $X_n=X$ are equal, we call this the countable $\ell_p$-sum of $X$ and denote it simply by $\pare{\bigoplus_\NN X}_p$. Many of our results will be ultimately based on Pe\l czy\'nski's decomposition theorem, or ``method'', in the following form: \textit{if $X\compl Y\compl X$ and $X\sim\pare{\bigoplus_\NN X}_p$ for $p\in\set{1,\infty}$ then $X\sim Y$}.

The symbol $M$ will stand for a metric space with metric $d$ (Banach spaces will be regarded as metric spaces endowed with the norm metric). Usually, we will assume that $M$ is pointed, that is, an arbitrary point $0\in M$ has been designated as a base point. The closed ball centered at $x\in M$ with radius $r>0$ will be denoted $B(x,r)$ or $B_M(x,r)$. We will consider the vector space $\Lip(M)$ of all Lipschitz functions $f:M\to\RR$, and the \emph{Lipschitz space} $\Lip_0(M)$ consisting of all $f\in\Lip(M)$ such that $f(0)=0$. The best Lipschitz constant of $f$, given by
$$
\lipnorm{f} = \sup\set{\frac{f(x)-f(y)}{d(x,y)} \,:\, x\neq y\in M} ,
$$
is a complete seminorm on $\Lip(M)$ and a complete norm on $\Lip_0(M)$. For any $x\in M$, we consider the evaluation functional $\delta(x)\in\Lip_0(M)^*$ taking $f\in\Lip_0(M)$ to $f(x)$. The space
$$
\lipfree{M} = \cl{\lspan}\,\set{\delta(x) \,:\, x\in M} \subset \Lip_0(M)^*
$$
is called the \emph{Lipschitz-free space} over $M$, and is in fact an isometric predual of $\Lip_0(M)$. The weak$^*$ topology induced by $\lipfree{M}$ agrees on $B_{\Lip_0(M)}$ with the topology of pointwise convergence. Thus, by the Banach-Dieudonn\'e theorem, an operator on, or to, $\Lip_0(M)$ is weak$^*$ continuous if and only if it is pointwise continuous.

The mapping $\delta:M\to\lipfree{M}$ is an isometric embedding. Lipschitz-free spaces satisfy the following linearization property: any Lipschitz map $\varphi:M\to X$ such that $\varphi(0)=0$ can be extended to an operator $\Phi:\lipfree{M}\to X$, in the sense that $\Phi(\delta(x))=\varphi(x)$ for all $x\in M$. Moreover, its norm $\norm{\Phi}$ equals the Lipschitz constant of $\varphi$. Because of this, $\lipfree{M}$ can be regarded as a canonical linearization of the metric space $M$. As a consequence, if two metric spaces $M_1$, $M_2$ are bi-Lipschitz equivalent with distortion $D\geq 1$, i.e. there exist a bijection $\varphi:M_1\to M_2$ and $c>0$ such that
$$
c\cdot d(\varphi(x),\varphi(y))\leq d(x,y)\leq cD\cdot d(\varphi(x),\varphi(y))
$$
for all $x,y\in M_1$, then $\lipfree{M_1}$ and $\lipfree{M_2}$ are $D$-isomorphic. In particular, isometric metric spaces have isometric Lipschitz-free spaces. Lipschitz-free spaces are also isometrically invariant with respect to changes of base point.

We will occasionally use the following product rule for Lipschitz constants:
\begin{equation}\label{eq:product rule}
\lipnorm{f\cdot g} \leq \lipnorm{f}\sup_{x\in\supp(f)}\abs{g(x)}+\lipnorm{g}\sup_{x\in\supp(g)}\abs{f(x)}
\end{equation}
for $f,g\in\Lip(M)$. It can be proved in a similar way as e.g. \cite[Proposition 1.30]{Weaver2}.

Throughout Section \ref{applications}, we will also need to use a number of auxiliary results on isomorphism or complementation relations between different Lipschitz and Lipschitz-free spaces that are now standard tools in the field. We list them all now, mostly without proof.

The first fact, due to C\'uth, Doucha and Wojtaszczyk \cite{CDW16}, is that every infinite-dimensional Lipschitz-free space contains a complemented copy of $\ell_1$.

\begin{lemma}[C\'uth, Doucha, Wojtaszczyk]\label{lm:cdw}
For any infinite metric space $M$, $\ell_1\compl\lipfree{M}$.
\end{lemma}

Next is Kalton's decomposition theorem, which we shall use in the following form.

\begin{lemma}[Kalton]\label{lm:kalton decomposition}
For any metric space $M$, $\lipfree{M} \compl \pare{\bigoplus_{n\in\NN} \lipfree{B(0,2^n)}}_1$.
\end{lemma}

\noindent This differs from Kalton's original formulation \cite[Proposition 4.3]{Kalton}, which merely claims that $\pare{\bigoplus_{k\in\ZZ} \lipfree{B(0,2^k)}}_1$ contains a subspace isomorphic to $\lipfree{M}$, but a similar argument yields Lemma \ref{lm:kalton decomposition} as follows. By \cite[Lemma 4.2]{Kalton}, there exist operators $T_k:\lipfree{M}\to\lipfree{B(0,2^k)}$, $k\in\ZZ$, and a universal constant $C<\infty$ such that $\mu=\sum_{k\in\ZZ}T_k\mu$ and $\sum_{k\in\ZZ}\norm{T_k\mu}\leq C\norm{\mu}$ for all $\mu\in\lipfree{M}$. Let $Y=\pare{\bigoplus_{n\in\NN} \lipfree{B(0,2^n)}}_1$ and consider the operators
\begin{align*}
T:\lipfree{M}\to Y \qquad &, \qquad T\mu=\big( \textstyle\sum_{k\leq 1}T_k\mu,T_2\mu,T_3\mu,\ldots \big) \\
S:Y\to\lipfree{M} \qquad &, \qquad S((\mu_n)_{n\in\NN})=\textstyle\sum_{n\in\NN}\mu_n
\end{align*}
then $ST$ is the identity, hence $TS$ is a projection of $Y$ onto $T(\lipfree{M})\sim\lipfree{M}$.

We will also need Kaufmann's theorem from \cite{Kaufmann}.

\begin{lemma}[Kaufmann]\label{lm:kaufmann}
For any Banach space $X$, $\lipfree{B_X} \sim \lipfree{X} \sim \pare{\bigoplus_\NN \lipfree{X}}_1$.
\end{lemma}

The next result has been stated by many authors in slightly different forms \cite{AACD2,Godard,HajekNovotny,Kaufmann}. We will need the following formulation, similar to \cite[Lemma 3.3]{Aliaga25}.

\begin{lemma}\label{lm:metric l1 sum}
Let $\lambda\geq 1$ and $\set{S_i:i\in I}$ be a family of non-empty, pairwise disjoint subsets of $M$ with $0\notin S_i$ and such that $d(x,0)+d(y,0)\leq \lambda\cdot d(x,y)$ for all $x,y$ belonging to different sets $S_i$. Then $\lipfree{\bigcup_{i\in I} S_i} \sim \pare{\bigoplus_{i\in I} \lipfree{S_i}}_1 \oplus \ell_1(I)$.
\end{lemma}

\noindent Following \cite{Aliaga25}, we say that sets $S_i$ satisfying the hypothesis of Lemma \ref{lm:metric l1 sum} are \emph{well-separated} (with respect to the base point).

Finally, we will need the following density result as well. Unlike the previous lemmas, this one is only valid for Lipschitz spaces, but not for Lipschitz-free spaces. It is a particular case of \cite[Lemma 1.3]{CCD}.

\begin{lemma}\label{lm:lip density}
Suppose that $(S_n)$ is a sequence of subsets of $M$ with the property that for every $x\in M$ there exist points $x_n\in S_n$ such that $x_n\to x$. Then $\Lip_0(M) \compl \pare{\bigoplus_{n\in\NN} \Lip_0(S_n)}_\infty$.
\end{lemma}

\section{Lipschitz extendable sets}\label{extensions}

Given any subset $S\subset M$ and any function $f\in\Lip(S)$, McShane's theorem (see e.g. \cite[Theorem 1.33]{Weaver2}) guarantees the existence of extensions $F\in\Lip(M)$ such that $F\restrict_S=f$ and $\lipnorm{F}=\lipnorm{f}$. Consequently, if $0\in S$ then one can identify $\lipfree{S}$ isometrically with the subspace $\cl{\lspan}\,\delta(S)$ of $\lipfree{M}$. However, McShane's theorem is purely existential and provides no information on the dependence of $F$ on $f$ as $f$ varies. The following stronger notions demand that there exist a linear dependence, possibly at the expense of increasing the Lipschitz constant. As we will see below, they are related to the complementability of $\lipfree{S}$ in $\lipfree{M}$.

\begin{definition}\label{def:lip extendable}
Let $M$ be a metric space and $S\subset M$.
\begin{enumerate}[label={\upshape{(\alph*)}}]
\item We say that $S$ is \emph{Lipschitz extendable in $M$} if there exists an extension operator $E:\Lip_0(S)\to\Lip_0(M)$. By \emph{extension operator} we mean an operator such that $(Ef)(x)=f(x)$ for all $f\in\Lip_0(S)$ and $x\in S$.
\item If $E$ can be chosen to be pointwise-to-pointwise continuous (hence weak$^*$-weak$^*$ continuous), we say that $S$ is \emph{weak$^*$ Lipschitz extendable} in $M$.
\item If $1\leq C<\infty$ and $E$ can be chosen such that $\norm{E}\leq C$, we say that $S$ is \emph{(weak$^*$) $C$-Lipschitz extendable} in $M$, or \emph{(weak$^*$) Lipschitz extendable with constant $C$}.
\end{enumerate}
\end{definition}

\begin{remark}
\label{remark:lip0_lip}
In Definition \ref{def:lip extendable}, we are implicitly assuming that $S$ contains the base point $0$ so that $\Lip_0(S)$ is correctly defined. In fact, the definition can be restated equivalently without using the base point: $S$ is (weak$^*$) $C$-Lipschitz extendable in $M$ if and only if there exists a (pointwise continuous) linear extension mapping $E':\Lip(S)\to\Lip(M)$ with $\lipnorm{E'f}\leq C\lipnorm{f}$. Indeed, if such a mapping exists and $0\in S$, then the restriction $E$ of $E'$ to $\Lip_0(M)$ witnesses Lipschitz extendability. Conversely, if $0\in S$ and there exists an extension operator $E:\Lip_0(S)\to\Lip_0(M)$, then we can take $E'f=f(0)+E(f-f(0))$ for any $f\in\Lip(S)$. This argument also shows that the choice of base point, whether it belongs to $S$ or not, does not affect Lipschitz extendability or its constant.
\end{remark}

In general, Lipschitz extendability of a set $S$ depends on the ambient metric space $M$. For instance, every set $S$ is extendable to one extra point (see Fact \ref{fact:one_point} below), but there are situations where no extension operator exists, see e.g. \cite{Naor15}. An important observation is that if $S$ is (weak$^*$) $C$-Lipschitz extendable to $M$ then it is also (weak$^*$) $C$-Lipschitz extendable to any subset $S'$ such that $S\subset S'\subset M$. This follows by composing extension operators with the restriction $R:f\in\Lip_0(M)\mapsto f\restrict_{S'}$, as $R$ is non-expansive and pointwise continuous.

Extension operators are a key tool in the study of the structure of Lipschitz and Lipschitz-free spaces due to the following known facts, brought to the forefront by \cite{AmbrosioPuglisi,HajekNovotny,Kaufmann,LancienPernecka}.
For part (c), recall that the Lipschitz extension constant $e(S,M)$ is the infimum of all $L<\infty$ such that, for every Banach space $X$ and every $1$-Lipschitz map $f:S\to X$, there is a $L$-Lipschitz map $F:M\to X$ extending $f$, that is, satisfying ${F\restrict}_S=f$ (with the convention that $e(S,M)=\infty$ if no such $L$ exists).

\begin{fact}\label{clipfact}
Let $M$ be a metric space, $S\subset M$ with $0\in S$, and $C\in [1,\infty)$.
\begin{enumerate}[label={\upshape{(\alph*)}}]
\item If $S$ is $C$-Lipschitz extendable in $M$, then $\Lip_0(S)$ is $C$-isomorphic to a $C$-complemented subspace of $\Lip_0(M)$.
\item $S$ is weak$^*$ $C$-Lipschitz extendable in $M$ if and only if $\lipfree{S}$ is $C$-complemented in $\lipfree{M}$.
\item The constant $e(S,M)$ equals
\begin{align*}
e(S,M) &= \inf\set{L\geq 1 \,:\, \text{$S$ is weak$^*$ $L$-Lipschitz extendable in $M$}} \\
&= \inf\set{\norm{P} \,:\, P:\lipfree{M}\to\lipfree{S}\text{ is a linear projection}} .
\end{align*}
\end{enumerate}
\end{fact}

\begin{proof}
(a) Suppose that there exists an extension operator $E:\Lip_0(S)\to\Lip_0(M)$ and let $R:\Lip_0(M)\to\Lip_0(S)$ denote the restriction operator $f\mapsto f\restrict_S$. Then the composition $RE$ is the identity, therefore $(ER)^2=ER$ is a projection with $\norm{ER}\leq\norm{E}$. Since $R$ is onto by McShane's theorem, the image of $ER$ is $E(\Lip_0(S))\sim\Lip_0(S)$.

(b) If, in (a), $E$ is pointwise-to-pointwise continuous then it has a preadjoint $E_*:\lipfree{M}\to\lipfree{S}$ which is a projection as $\duality{f,E_*(\delta(x))}=\duality{Ef,\delta(x)}=f(x)$ for all $x\in S$ and $f\in\Lip_0(S)$. Conversely, if there is a projection $P$ of $\lipfree{M}$ onto $\lipfree{S}$ then $P^*:\Lip_0(S)\to\Lip_0(M)$ is an extension operator as $(P^*f)(x)=\duality{P^*f,\delta(x)}=\duality{f,P(\delta(x))}=\duality{f,\delta(x)}=f(x)$ for all $f\in\Lip_0(S)$ and $x\in S$.

(c) The equality of both infima follows from (b). We will prove that $e(S,M)$ equals the second one.

Let $L>e(S,M)$. Since the map $\delta:S\to\lipfree{S}$ is an isometry, it can be extended to a $L$-Lipschitz map $\Delta:M\to\lipfree{S}$. By the linearization property of Lipschitz-free spaces, there is an operator $P:\lipfree{M}\to\lipfree{S}$ with $\norm{P}\leq L$ that extends $\Delta$, i.e. $P(\delta(x))=\Delta(x)$ for all $x\in M$. In particular $P(\delta(x))=\delta(x)$ for $x\in S$. Since $\delta(S)$ is linearly dense in $\lipfree{S}$, we conclude that $P$ is a projection onto $\lipfree{S}$.

Conversely, suppose that there exists a projection $P:\lipfree{M}\to\lipfree{S}$ with $\norm{P}\leq L$, and let $f:S\to X$ be a $1$-Lipschitz map into a Banach space $X$. We may assume that $f(0)=0$. By the linearization property, there exists an operator $F:\lipfree{S}\to X$ with $\norm{F}\leq 1$ and $F(\delta(x))=f(x)$ for $x\in S$. Then $g=F\circ P\circ \delta:M\to X$ is an $L$-Lipschitz extension of $f$, as $g(x)=F(P(\delta(x)))=F(\delta(x))=f(x)$ for all $x\in S$. So $e(S,M)\leq L$.
\end{proof}

Thus, in our notation, $e(S,M)$ is the least constant $C$ such that $S$ is weak$^*$ $C$-Lipschitz extendable in $M$, similarly to how $\lambda(S,M)$ is the least constant $C$ such that $S$ is $C$-Lipschitz extendable in $M$. At the time of writing, we do not know whether both constants are equivalent in general.

\medskip

Lipschitz extendable sets used in the Lipschitz-free space literature belong almost exclusively to the following two classes:
\begin{itemize}
\item Lipschitz retracts. If there exists a $C$-Lipschitz retraction $r:M\to S$ (i.e. a mapping with $r(x)=x$ for all $x\in S$), then $S$ is weak$^*$ $C$-Lipschitz extendable in $M$ as witnessed by the extension operator $f\mapsto f\circ r$.
\item Metric spaces with finite Nagata dimension, including doubling spaces (such as subsets of $\RR^n$) and ultrametric spaces. We will not define these notions here, and refer instead to \cite[Section 3]{FG} for the definitions and the proof that these sets are always weak$^*$ Lipschitz extendable. In this case extendability is absolute, that is, it does not depend on the ambient space $M$. 
\end{itemize}

The former class already yields a simple case where extendability is guaranteed, that will be needed later.

\begin{fact}\label{fact:one_point}
Let $M$ be a metric space and $x_0\in M$. Then $M\setminus\set{x_0}$ is weak$^*$ $(2+\ep)$-Lipschitz extendable in $M$ for every $\ep>0$.
\end{fact}

\begin{proof}
Put $M_0=M\setminus\set{x_0}$. If $x_0$ is an accumulation point of $M$ then every $f\in\Lip(M_0)$ admits a unique extension to a function in $\Lip(M)$ with the same Lipschitz constant, so extendability holds with constant $1$. Otherwise $D:=d(x_0,M_0)>0$, and we may pick $y_0\in M_0$ such that $d(x_0,y_0)<D(1+\ep)$. Define a mapping $r:M\to M$ by $r(x_0)=y_0$ and $r(x)=x$ for $x\neq x_0$. Then $r$ is a retraction onto $M_0$, and
$$
d(r(x),r(x_0)) = d(x,y_0) \leq d(x,x_0)+d(x_0,y_0) < d(x,x_0)+D(1+\ep) \leq d(x,x_0)\cdot(2+\ep)
$$
for $x\in M_0$. Thus $M_0$ is a $(2+\ep)$-Lipschitz retract of $M$ and this yields the conclusion.
\end{proof}

Our goal in this section is to determine new classes of Lipschitz extendable sets. So we ask a rather natural question: if two sets are Lipschitz extendable, does their union have the same property? Since we have a different extension operator from each of them, can we combine them both somehow? If the sets are separated from each other, it is possible to do this in a straightforward manner.

\begin{lemma}\label{lm:lip extendable union}
Let $S_1,S_2$ be (weak$^*$) Lipschitz extendable subsets of $M$. If $d(S_1,S_2)>0$ and $S_2$ is bounded, then $S_1\cup S_2$ is (weak$^*$) Lipschitz extendable in $M$.
\end{lemma}

\begin{proof}
We may assume that the sets are non-empty. Fix a base point $0\in S_1$, as we may, and set $r=d(S_1,S_2)$. By Remark \ref{remark:lip0_lip}, there exist extension operators $E_1:\Lip_0(S_1)\to\Lip_0(M)$ and $E_2:\Lip(S_2)\to\Lip(M)$ with norms $C_1,C_2<\infty$, respectively. Put
$$
h(x)=\max\set{0,1-\frac{1}{r}d(x,S_2)}
$$
for $x\in M$, so that $h\in\Lip_0(M)$, $0\leq h\leq 1$, $h=0$ on $S_1$, $h=1$ on $S_2$, and $h$ has bounded support. Then, for $f\in\Lip_0(S_1\cup S_2)$, define a function $Ef:M\to\RR$ by
$$
Ef = E_1(f\restrict_{S_1}) + h\cdot\pare{ E_2(f\restrict_{S_2}) - E_1(f\restrict_{S_1}) } .
$$
The map $E:f\mapsto Ef$ is clearly linear, and $Ef$ is an extension of $f$ by the properties of $h$. Moreover, if $E_1,E_2$ are pointwise continuous then so is $E$. So we only need to estimate the Lipschitz constant of $Ef$. Denote $F_i=E_i(f\restrict_{S_i})$ for simplicity, and note $\lipnorm{F_i}\leq C_i\lipnorm{f}$. Then, using \eqref{eq:product rule}, we estimate
\begin{align*}
\lipnorm{Ef} &\leq \lipnorm{F_1} + \lipnorm{h\cdot (F_2-F_1)} \\
&\leq C_1\lipnorm{f} + \lipnorm{h}\sup_{x\in\supp(h)}\abs{F_2(x)-F_1(x)} + \norm{h}_\infty\lipnorm{F_2-F_1} \\
&\leq (2C_1+C_2)\lipnorm{f} + \frac{1}{r}\sup_{x\in\supp(h)}\pare{\abs{F_1(x)}+\abs{F_2(x)}} .
\end{align*}
Given $x\in\supp(h)$ we have
$$
\abs{F_1(x)} \leq \lipnorm{F_1}d(x,0) \leq C_1\lipnorm{f}\cdot(d(S_2,0)+\diam(S_2)+r)
$$
and, fixing an arbitrary point $x_0\in S_2$,
\begin{align*}
\abs{F_2(x)} &\leq \lipnorm{F_2}d(x,x_0) + \abs{f(x_0)} \\
&\leq C_2\lipnorm{f}\cdot(\diam(S_2)+r) + \lipnorm{f}d(x_0,0) \\
&\leq C_2\lipnorm{f}\cdot(\diam(S_2)+r) + \lipnorm{f}\cdot(d(S_2,0)+\diam(S_2)) .
\end{align*}
Combining all of the above yields $\lipnorm{Ef}\leq C\lipnorm{f}$ for some $C<\infty$, as was to be shown.
\end{proof}

Repeated application of Lemma \ref{lm:lip extendable union} yields, by induction, a similar result for the union of finitely many Lipschitz extendable sets.

\begin{proposition}
Let $S_1,\ldots,S_n\subset M$ be (weak$^*$) Lipschitz extendable subsets of $M$. Suppose that $d(S_i,S_j)>0$ for all $i\neq j$, and at most one of the sets $S_i$ is unbounded. Then $S_1\cup\ldots\cup S_n$ is (weak$^*$) Lipschitz extendable in $M$.
\end{proposition}

This argument does not extend (without modifications) to the infinite case because the constant of Lipschitz extendability could increase without bound. But it is possible to make the same general construction work for infinite unions if we assume that they are Lipschitz extendable with the same constant and place some additional geometrical constraints. The following theorem shows that this is the case when the sets are ``sufficiently small'' and ``sufficiently separated'' relative to some anchor point $x_0\in M$. While these conditions seem technical, they will be sufficient for our applications in the following section.

\newpage
\begin{theorem}\label{th:lip extendable infinite union}
Let $M$ be a complete metric space, $x_0\in M$, and $C,D,\lambda<\infty$. Suppose that $\set{S_i:i\in I}$ is a family of non-empty closed subsets of $M$ not containing $x_0$ and satisfying:
\begin{enumerate}[label={\upshape{(\roman*)}}]
\item each $S_i$ is (weak$^*$) $C$-Lipschitz extendable in $M$,
\item $\diam(S_i)\leq D\cdot d(S_i,x_0)$ for all $i$, and
\item $d(x,x_0)+d(y,x_0)\leq \lambda\cdot d(x,y)$ for all $x,y$ belonging to different sets $S_i$.
\end{enumerate}
Then, the union $\bigcup_{i\in I} S_i$ is (weak$^*$) $K$-Lipschitz extendable in $M$ for a constant $K<\infty$ that depends only on $C,D,\lambda$.
\end{theorem}

Our argument yields $K\leq \text{constant}\cdot CD^2\lambda^2$ (assuming $D\geq 1$), but our applications will not need the precise value. Recall that condition (iii) means that the sets $S_i$ are well-separated, as in Lemma \ref{lm:metric l1 sum}.

\begin{proof}
By Fact \ref{fact:one_point}, the set $\bigcup_{i\in I} S_i$ is weak$^*$ extendable in $\bigcup_{i\in I} S_i \cup\set{x_0}$ with a fixed constant. Therefore, it suffices to prove that the set $S=\bigcup_{i\in I} S_i \cup\set{x_0}$ is extendable in $M$ with the conditions in the theorem statement.

Recall that the choice of base point does not affect Lipschitz extendability or its constant. Thus, we may assume that $x_0$ is the base point of $M$ and, again by Remark \ref{remark:lip0_lip}, that for each $i\in I$ there exists a linear extension mapping $E_i:\Lip(S_i)\to\Lip(M)$ such that $\lipnorm{E_if}\leq C\lipnorm{f}$ for any $f\in\Lip(S_i)$.

For each $i\in I$, put $r_i=d(S_i,x_0)/2\lambda>0$ and $U_i=\set{x\in M : d(x,S_i)\leq r_i}$, and define a function $\Pi_i\in\Lip(M)$ by
$$
\Pi_i(x) = \max\set{0,1-\frac{1}{r_i}d(x,S_i)}
$$
for $x\in M$. Taking infima in (iii), we obtain
$$
d(S_i,x_0)+d(S_j,x_0) \leq \lambda d(S_i,S_j)
$$
for all $i\neq j\in I$. This guarantees $d(S_i,S_j)>0$. Moreover, $U_i$ and $U_j$ are disjoint: if there existed $x\in U_i\cap U_j$, we would have
$$
d(S_i,S_j) \leq d(x,S_i)+d(x,S_j) \leq r_i+r_j = \frac{1}{2\lambda}(d(S_i,x_0)+d(S_j,x_0)) \leq \frac{1}{2}d(S_i,S_j)
$$
by (iii), contradicting $d(S_i,S_j)>0$. We also have $x_0\notin U_i$ as (iii) forces $\lambda\geq 1$.

Now fix $f\in\Lip_0(S)$. For each $i\in I$, let $f_i:M\to\RR$ be given by
$$
f_i(x) = \Pi_i(x)\cdot(E_i f\restrict_{S_i})(x)
$$
for $x\in M$. Finally, define a function $Ef:M\to\RR$ by
$$
Ef(x) = \begin{cases}
f_i(x) &\text{, if } x\in U_i \\
0 &\text{, otherwise}
\end{cases}
$$
for $x\in M$. Note that $Ef$ is well defined because the sets $U_i$ are pairwise disjoint. Moreover, each $f_i$ is an extension of $f\restrict_{S_i}$ because $\Pi_i(x)=1$ for $x\in S_i$, and $f(x_0)=0$ as $x_0\notin U_i$ for all $i$, thus $Ef$ is an extension of $f$. It is also clear that the map $f\mapsto Ef$ is pointwise continuous if every operator $E_i$ is. Thus, it only remains to be proved that $\lipnorm{Ef}\leq K\lipnorm{f}$ for some $K$ depending only on $C,D,\lambda$; this will show that $E:\Lip_0(S)\to\Lip_0(M)$ is the desired extension operator.

We start by estimating the Lipschitz constant of $f_i$. Fix a point $p_i\in S_i$. Then, for $x\in U_i$ we have
\begin{align*}
\abs{E_i(f\restrict_{S_i})(x)} &\leq \abs{E_i(f\restrict_{S_i})(p_i)} + \lipnorm{E_i(f\restrict_{S_i})}d(x,p_i) \\
&\leq \abs{f(p_i)} + C\lipnorm{f}d(x,p_i) \\
&\leq \lipnorm{f}d(p_i,x_0) + C\lipnorm{f}d(x,p_i) \\
&\leq \lipnorm{f}(d(S_i,x_0)+\diam(S_i)) + C\lipnorm{f}(r_i+\diam(S_i)) \\
&\leq \lipnorm{f}\cdot\pare{1+D + \frac{C}{2\lambda} + CD}\cdot d(S_i,x_0)
\end{align*}
Since $f_i$ vanishes outside of $U_i$, we can use \eqref{eq:product rule} to estimate
\begin{align*}
\lipnorm{f_i} &\leq \norm{\Pi_i}_\infty\lipnorm{E_i(f\restrict_{S_i})} + \lipnorm{\Pi_i}\sup\set{\abs{E_i(f\restrict_{S_i})(x)} : x\in U_i} \\
&\leq 1\cdot C\lipnorm{f} + \frac{1}{r_i}\cdot \lipnorm{f}\cdot\pare{1+D + \frac{C}{2\lambda} + CD}\cdot d(S_i,x_0) \\
&= \lipnorm{f}\cdot\underbrace{\pare{2C+2\lambda(1+D+CD)}}_{K'} .
\end{align*}

Now let $x,y\in M$. If none of them belong to any $U_i$, then $Ef(x)=Ef(y)=0$. If both belong to the same set $U_i$, or one of them (say $x$) does and the other belongs to none, then $\abs{Ef(x)-Ef(y)}=\abs{f_i(x)-f_i(y)}\leq K'\lipnorm{f}d(x,y)$. Finally, suppose $x\in U_i$ and $y\in U_j$ for $i\neq j$. Fix points $p_i\in S_i$ and $p_j\in S_j$. Then we estimate
\begin{align*}
\abs{Ef(x)-Ef(y)} &\leq \abs{Ef(x)-Ef(p_i)} + \abs{Ef(p_i)-Ef(p_j)} + \abs{Ef(p_j)-Ef(y)} \\
&= \abs{f_i(x)-f_i(p_i)} + \abs{f(p_i)-f(p_j)} + \abs{f_j(p_j)-f_j(y)} \\
&\leq K'\lipnorm{f}(d(x,p_i)+d(p_j,y)) + \lipnorm{f}d(p_i,p_j) \\
&\leq \lipnorm{f}(K'(r_i+\diam(S_i)+r_j+\diam(S_j)) + d(p_i,x_0)+d(p_j,x_0)) \\
&\leq \lipnorm{f}(K'(r_i+r_j)+(K'+1)(\diam(S_i)+\diam(S_j)) + d(S_i,x_0)+d(S_j,x_0)) \\
&\leq \lipnorm{f}(K'(r_i+r_j)+((K'+1)D+1)(d(S_i,x_0)+d(S_j,x_0)) \\
&= \lipnorm{f}\underbrace{(K'+2\lambda((K'+1)D+1))}_K (r_i+r_j) \\
&= K\lipnorm{f}(2(r_i+r_j)-r_i-r_j) \\
&\leq K\lipnorm{f}\pare{\frac{1}{\lambda}(d(S_i,x_0)+d(S_j,x_0)) - d(x,S_i)-d(y,S_j)} \\
&\leq K\lipnorm{f}\pare{d(S_i,S_j) - d(x,S_i)-d(y,S_j)} \\
&\leq K\lipnorm{f}d(x,y) .
\end{align*}
This ends the proof.
\end{proof}

\section{Applications}\label{applications}

In this section, we will apply the ideas from Section \ref{extensions}, in particular Theorem \ref{th:lip extendable infinite union}, to obtain new results concerning the isomorphic structure and classification of Lipschitz and Lipschitz-free spaces. The section is split into separate subsections, each focusing on a different object or notion. Additional background and required results are provided in each subsection.

\subsection{Homogeneous metric spaces}
\label{sec:homogeneous}

A metric space $M$ is \emph{(metrically) homogeneous} if, given any $x,y\in M$, there is a bijective isometry $\varphi:M\to M$ such that $\varphi(x)=y$. In other words, the isometry group of $M$ acts transitively on $M$. For instance, metric groups with an invariant distance are homogeneous, as translations are isometries. Homogeneity allows us to ``move sets around'' in $M$ in such a way that geometric separation constraints such as the ones required in Theorem \ref{th:lip extendable infinite union} are satisfied. That is the core observation behind the following result.

\begin{proposition}\label{pr:homogeneous unbounded sequence}
Let $M$ be homogeneous and unbounded. Let $C\in [1,\infty)$, and suppose that $(S_n)$ is a sequence of non-empty bounded subsets of $M$ that are weak$^*$ $C$-Lipschitz extendable in $M$. Then
$$
\pare{\bigoplus_\NN\bigoplus_{n=1}^\infty\lipfree{S_n}}_1\compl\lipfree{M} .
$$
\end{proposition}

\begin{proof}
We will prove that $\pare{\bigoplus_n\lipfree{S_n}}_1\compl\lipfree{M}$. That is enough: if this holds, then the general statement is obtained by applying that case to a sequence $(S'_k)$ that contains countably many copies of each $S_n$ (note that we are not assuming the sets $S_n$ to be disjoint).

Fix a base point $0$ of $M$. We now build a sequence $(\mathcal{A}_n)$ of subsets of $M$ and an increasing sequence $(k_n)$ in $\NN$ iteratively. Start with $\mathcal{A}_0=\set{0}$ and $k_0=0$. For $n\geq 1$, once sets $\mathcal{A}_i$ and numbers $k_i$ have been chosen for $i<n$, we choose $k_n$ satisfying
\begin{itemize}
\item $k_n\geq k_{n-1}+4$,
\item $2^{k_n}>\diam(S_n)$, and
\item there exists some point $p_n\in M$ with $2^{k_n+1}\leq d(p_n,0)\leq 2^{k_n+2}$.
\end{itemize}
Fix some $q_n\in S_n$. By homogeneity, there exists a bijective isometry $\varphi:M\to M$ such that $\varphi(q_n)=p_n$. Let $\mathcal{A}_n=\varphi(S_n)$. Then $\mathcal{A}_n$ is isometric to $S_n$ and weak$^*$ $C$-Lipschitz extendable in $M$: if $E$ is the extension operator associated to $S_n$ then $E'f=E(f\circ\varphi)\circ\varphi^{-1}$ defines an extension operator for $f\in\Lip(\mathcal{A}_n)$.

Note that $2^{k_n}\leq d(x,0)\leq 2^{k_n+3}$ for all $x\in \mathcal{A}_n$. Suppose that $x\in\mathcal{A}_n$ and $y\in\mathcal{A}_m$ with $n>m\geq 1$. Then we have
\begin{align*}
d(x,0) + d(y,0) &\leq 2^{k_n+3} + 2^{k_m+3} = 8\cdot (2^{k_n}+2^{k_m}) \\
&\leq 16\cdot 2^{k_n} = 32\cdot (2^{k_n}-2^{k_n-1}) \\
&\leq 32\cdot (2^{k_n}-2^{k_m+3}) \leq 32 (d(x,0)-d(y,0)) \leq 32\,d(x,y) .
\end{align*}
We have thus shown that the sequence $(\mathcal{A}_n)_{n\geq 1}$ satisfies the hypotheses of Theorem \ref{th:lip extendable infinite union} with $D=1$ and $\lambda=32$. Therefore the union $\mathcal{A}=\bigcup_{n=1}^\infty\mathcal{A}_n$ is weak$^*$ Lipschitz extendable in $M$, and we deduce $\lipfree{\mathcal{A}}\compl\lipfree{M}$. Moreover, the sets $\mathcal{A}_n$ are non-empty and pairwise disjoint, so Lemma \ref{lm:metric l1 sum} implies
$$
\pare{\bigoplus_{n=1}^\infty\lipfree{S_n}}_1 \equiv \pare{\bigoplus_{n=1}^\infty\lipfree{\mathcal{A}_n}}_1 \compl \pare{\bigoplus_{n=1}^\infty\lipfree{\mathcal{A}_n}}_1\oplus\ell_1 \sim \lipfree{\mathcal{A}} \compl \lipfree{M}
$$
as we needed to prove.
\end{proof}

As a consequence, we obtain the following generalization of Kaufmann's theorem (Lemma \ref{lm:kaufmann}) from Banach spaces to certain homogeneous spaces.

\begin{theorem}\label{th:homogeneous unbounded extendable}
Let $M$ be homogeneous and unbounded. Suppose that there is $C\in [1,\infty)$ such that all closed balls in $M$ are weak$^*$ $C$-Lipschitz extendable in $M$. Then $\lipfree{M}\sim\pare{\bigoplus_\NN\lipfree{M}}_1$.
\end{theorem}

\begin{proof}
On one hand, by Lemma \ref{lm:kalton decomposition} we have
$$
\lipfree{M} \compl \pare{\bigoplus_{n\in\NN}\lipfree{B(0,2^n)}}_1 \compl \pare{\bigoplus_\NN\bigoplus_{n\in\NN}\lipfree{B(0,2^n)}}_1 .
$$
On the other hand, we have by Proposition \ref{pr:homogeneous unbounded sequence}
$$
\pare{\bigoplus_\NN\bigoplus_{n\in\NN}\lipfree{B(0,2^n)}}_1 \compl \lipfree{M} .
$$
As the space $\pare{\bigoplus_\NN\bigoplus_{n\in\NN}\lipfree{B(0,2^n)}}_1$ is clearly isometric to its countable $\ell_1$-sum, Pe\l czy\'nski's decomposition method shows that it is isomorphic to $\lipfree{M}$, and this ends the proof.
\end{proof}

\begin{remark}
In Proposition \ref{pr:homogeneous unbounded sequence} or Theorem \ref{th:homogeneous unbounded extendable}, if we only assume extendability rather than weak$^*$ extendability then, using the same argument, we obtain the corresponding conclusion for Lipschitz spaces and $\ell_\infty$-sums in place of Lipschitz-free spaces and $\ell_1$-sums. For instance, in Theorem \ref{th:homogeneous unbounded extendable} we obtain $\Lip_0(M)\sim\pare{\bigoplus_\NN\Lip_0(M)}_\infty$.
\end{remark}

We do not know of any infinite-dimensional Lipschitz-free space $\lipfree{M}$ that does not satisfy the conclusion of Theorem \ref{th:homogeneous unbounded extendable}.

\begin{question}\label{q:l1 sum}
Does $\lipfree{M}\sim\pare{\bigoplus_\NN\lipfree{M}}_1$ hold for every infinite metric space $M$? Does it hold under the additional assumption that $M$ is homogeneous?
\end{question}

\subsection{Nets in Banach spaces}
\label{sec:nets}

A subset $N$ of $M$ is called a \emph{net in $M$}, or a $(\varepsilon,\delta)$-net if we need to be more precise, if there exist $\varepsilon,\delta>0$ such that $N$ is
\begin{itemize}
\item \emph{$\varepsilon$-dense in $M$}, that is, for all $x\in M$ there exists $y\in N$ with $d(x,y)\leq\varepsilon$, and
\item \emph{$\delta$-separated}, that is, $d(x,y)\geq\delta$ for all $x\neq y\in N$.
\end{itemize}

We will now apply the previous results to address several open problems concerning the Lipschitz and Lipschitz-free spaces over a net $N_X$ in a Banach space $X$. We start by highlighting an important observation by H\'ajek and Novotn\'y according to which the choice of $N_X$ is irrelevant if we are only concerned with isomorphic structure \cite[Proposition 5]{HajekNovotny}.

\begin{lemma}[H\'ajek, Novotn\'y]\label{lm:hajek novotny}
If $N_1,N_2$ are nets in $M$ with the same cardinality, then $\lipfree{N_1}\sim\lipfree{N_2}$.
\end{lemma}

First, we consider Question \ref{q:l1 sum} for the particular case of nets in Banach spaces. This has been asked in \cite[Question 6.5]{AACD3} in a more general framework; the (weaker) dual question has been posed in \cite[Question 4]{CCD}. The question is also implicit in \cite{HajekNovotny}, where it is answered under additional assumptions on the Banach space $X$. We will now give an affirmative answer for all $X$. We shall use the fact that balls in nets, like balls in Banach spaces, are Lipschitz retracts.

\begin{lemma}\label{lm:net ball retract}
Let $X$ be a Banach space and $N_X$ a net in $X$. Then there exists $C\in [1,\infty)$ such that, for every closed ball $B$ in $X$, the set $B\cap N_X$ is either empty or a $C$-Lipschitz retract of $N_X$.
\end{lemma}

\begin{proof}
Suppose that $N_X$ is a $(\varepsilon,\delta)$-net in $X$ for some $\varepsilon,\delta>0$. Let $B=B_X(p,R)$ be a closed ball in $X$, where $p\in X$ and $R>0$, and suppose that $B\cap N_X\neq\varnothing$. We claim that $B\cap N_X$ is a $(2\varepsilon,\delta)$-net in $B$. Indeed, if $R\leq\varepsilon$ then this is clear as $\diam(B)\leq 2\varepsilon$ and $B\cap N_X$ is non-empty. On the other hand, if $R>\varepsilon$ then for every $x\in B$ there exists $y\in B(p,R-\varepsilon)$ with $\norm{x-y}\leq\varepsilon$, and also $z\in N_X$ with $\norm{y-z}\leq\varepsilon$. Thus $\norm{x-z}\leq 2\varepsilon$ and $z\in B$ as $\norm{p-z}\leq\norm{p-y}+\norm{y-z}\leq R$.

Now let $r:X\to B$ be the radial projection given by
$$
r(x) = \begin{cases}
p+R\cdot\dfrac{x-p}{\norm{x-p}} &\text{, if $x\notin B$} \\
x &\text{, if $x\in B$}
\end{cases}
$$
which is known to be a $2$-Lipschitz retraction, and let $\pi:B\to B\cap N_X$ be any nearest-point map. Put $\varphi:=\pi\circ r\restrict_{N_X}:N_X\to B\cap N_X$. For $x\in B\cap N_X$ we have $r(x)=x$ (as $x\in B$) and $\pi(x)=x$ (as $x\in N_X$), thus $\varphi$ is a retraction onto $B\cap N_X$. If $x\neq y\in N_X$, then
\begin{align*}
\norm{\varphi(x)-\varphi(y)} &\leq \norm{\pi(r(x))-r(x)} + \norm{r(x)-r(y)} + \norm{r(y)-\pi(r(y))} \\
&\leq 2\varepsilon + 2\norm{x-y} + 2\varepsilon \leq \pare{2+4\frac{\varepsilon}{\delta}}\norm{x-y}
\end{align*}
as $\norm{x-y}\geq\delta$. Thus the result holds with $C=2+4\frac{\varepsilon}{\delta}$, independent of $B$.
\end{proof}

\begin{theorem}\label{th:net l1 sum}
Let $X$ be a Banach space and $N_X$ a net in $X$. Then $\lipfree{N_X}\sim\pare{\bigoplus_\NN\lipfree{N_X}}_1$.
\end{theorem}

\begin{proof}
By Lemma \ref{lm:hajek novotny}, it is enough to prove the theorem for one specific net $N_X$ in $X$. We may choose a net $N_X$ that is an additive subgroup of $X$: if $X$ is infinite-dimensional then such a net exists by \cite[Theorem C]{DBRS}, otherwise we may assume $X=\RR^n$ for some $n\in\NN$ and take $N_X=\ZZ^n$. Then $N_X$ is homogeneous as a metric space, as translations in $X$ are isometries. Moreover, by Lemma \ref{lm:net ball retract} there exists $C<\infty$ such that every ball in $N_X$ is a $C$-Lipschitz retract of $N_X$, hence weak$^*$ $C$-Lipschitz extendable in $N_X$. The result now follows by applying Theorem \ref{th:homogeneous unbounded extendable}.
\end{proof}

The second problem we consider is whether $\Lip_0(N_X)$ must be isomorphic to $\Lip_0(X)$. 
Note that this question only makes sense for Lipschitz spaces, as $\lipfree{N_X}$ is never isomorphic to $\lipfree{X}$: the former has the Radon-Nikod\'ym property and the latter does not, e.g. by \cite[Theorem C]{AGPP}.
The question was first posed in \cite[Question 3]{CCD} after proving that it holds when $X$ is finite-dimensional. In fact, in that setting we even have $\Lip_0(M)\sim\Lip_0(X)$ whenever $M$ is merely non-porous in $X$ \cite{Aliaga25}. However, finite dimensionality is essential in these arguments, via absolute Lipschitz extendability. 

While we cannot fully answer the question, we provide a partial answer by showing that $\Lip_0(X)\compl\Lip_0(N_X)$ for every Banach space $X$. This was already shown in \cite[Corollary 3.5]{CCD} under additional assumptions on $X$. The question is thus reduced to determining whether the reverse complementation relation $\Lip_0(N_X)\compl\Lip_0(X)$ holds; if it does, isomorphism then follows by Pe\l czy\'nski's method.

\begin{theorem}\label{th:lip net complemented}
Let $X$ be a Banach space and $N_X$ a net in $X$. Then $\Lip_0(X)\compl\Lip_0(N_X)$.
\end{theorem}

\begin{proof}
We make a construction similar to that of Proposition \ref{pr:homogeneous unbounded sequence} using a sequence of balls in $X$. For $n\in\NN$, let $S_n=B_X(x_n,2^n)$ where $x_n\in X$ is chosen such that $d(x_n,0)=2^{n+2}$. We claim that the sets $S_n\cap N_X$ satisfy the hypotheses of Theorem \ref{th:lip extendable infinite union} with respect to the metric space $N_X$. Indeed, weak$^*$ $C$-Lipschitz extendability with a uniform constant $C$ is guaranteed by Lemma \ref{lm:net ball retract}, and condition (ii) is satisfied with $D=1$. Given $x\in S_n$ and $y\in S_m$ with $n>m$, we have
\begin{align*}
d(x,0) + d(y,0) &\leq 2^{n+2}+2^n + 2^{m+2}+2^m < 10\cdot 2^n = 20\cdot (2^{n+2} - 2^n - 2^{n+1} - 2^{n-1}) \\
&\leq 20\cdot (2^{n+2} - 2^n - 2^{m+2} - 2^m) \leq 20\,(d(x,0)-d(y,0)) \leq 20\,d(x,y)
\end{align*}
thus condition (iii) is satisfied with $\lambda=20$. It follows that $\pare{\bigoplus_{n\in\NN}\lipfree{S_n\cap N_X}}_1 \compl \lipfree{N_X}$.

On the other hand, we may repeat the construction from \cite[Proposition 3.5]{Aliaga25} to show that $\Lip_0(B_X) \compl \pare{\bigoplus_{n\in\NN}\Lip_0(S_n\cap N_X)}_\infty$. Indeed, let $M_n=2^{-n}(S_n\cap N_X-x_n)$ (that is, we translate and rescale the set $S_n\cap N_X$ so that it becomes a subset of $B_X$). Because $N_X$ is $\varepsilon$-dense in $X$ for some $\varepsilon>0$, $S_n\cap N_X$ is $2\varepsilon$-dense in $S_n$ for $n$ large enough (by the same argument as in the proof of Lemma \ref{lm:net ball retract}), therefore $M_n$ is $2^{-(n-1)}\varepsilon$-dense in $B_X$ for $n$ large enough, and Lemma \ref{lm:lip density} implies that $\Lip_0(B_X) \compl \pare{\bigoplus_{n\in\NN}\Lip_0(M_n)}_\infty$. However, $\Lip_0(M_n)$ is isometric to $\Lip_0(S_n\cap N_X)$ as Lipschitz spaces are invariant under dilations, so this finishes the proof of our claim. We conclude by recalling that $\lipfree{X}\sim\lipfree{B_X}$ by Lemma \ref{lm:kaufmann}.
\end{proof}

\subsection{The integer grid in \texorpdfstring{$\ell_1$}{l1}}
\label{sec:grids}

Next, we consider the case of \emph{grids} in a Banach space $X$, that is, sets of the form
$$
\ZZ_X = \set{x\in X \,:\, x=\sum_{n=1}^N a_ne_n \text{ with $a_n\in\ZZ$ for all $n$}}
$$
where $(e_n)$ is a fixed Schauder basis of $X$. We will focus on the case where $(e_n)$ is the standard basis in $X=\ell_1$. The grid $\ZZ_{\ell_1}$ is an additive subgroup of $\ell_1$ but not a net in $\ell_1$; therefore, we may apply the results from Section \ref{sec:homogeneous} to it but not those from Section \ref{sec:nets}.

The particular geometry of $\ell_1$ allows for the construction of relatively simple extension operators. This was exploited in \cite{HajekPernecka,LancienPernecka} to construct Schauder bases in $\lipfree{\ell_1^n}$ and $\lipfree{\ell_1}$. We will begin by listing some of these sets admitting extension operators. For a subset $I\subset\NN$, we shall denote by $\ell_1(I)$ the subspace of $\ell_1$ consisting of elements whose support is contained in $I$, and identify $\ell_1^n$ with $\ell_1(\set{1,2,\ldots,n})$ for $n\in\NN$.

The following lemma is well known to specialists, although it cannot be found explicitly in the existing literature.

\newpage
\begin{lemma}
\label{lm:lemma_extension_Z_ell1}
The following subsets of $\ell_1$ are weak$^*$ $1$-Lipschitz extendable in $\ell_1$:
\begin{enumerate}[label={\upshape{(\alph*)}}]
\item $\ell_1(I)$ for $I\subset\NN$,
\item $Q_r:=\set{x\in \ell_1 \,:\, \abs{x_n}\leq r \text{ for all $n$}}$ for $r>0$,
\item $\ZZ_{\ell_1}$,
\item $\ZZ_{\ell_1(I)}:=\ZZ_{\ell_1}\cap\ell_1(I)$,
\item $\ZZ_{\ell_1}\cap Q_r$.
\end{enumerate}
\end{lemma}

\begin{proof}
Case (a) is clear as the standard linear projection $P_I:\ell_1\to\ell_1(I)$ is a $1$-Lipschitz retraction. For case (b), the $1$-Lipschitz retraction $R_r:\ell_1\to Q_r$ is constructed by retracting coordinatewise onto $[-r,r]\subset\RR$.

For the remaining cases, we require the following construction by Lancien and Perneck\'a \cite{LancienPernecka}. Fix $n\in\NN$ and let $Q$ be a unit hypercube in $\ell_1^n$ with vertices $V_Q$ in $\ZZ_{\ell_1}$. More precisely, fix $v\in\ZZ_{\ell_1^n}$ and let $Q=\conv(V_Q)$ with set of vertices $V_Q=\set{v_\gamma\,:\,\gamma\in\set{0,1}^n}$ where we denote $v_\gamma=v+\sum_{i=1}^n\gamma_ne_n$. Then, for any function $f:V_Q\to\RR$, we consider the interpolation function $\Lambda_Qf:Q\to\RR$ given by
$$
\Lambda_Qf(x) = \sum_{\gamma\in\set{0,1}^n} f(v_\gamma)\cdot\prod_{i=1}^n \pare{1-\gamma_i+(-1)^{1-\gamma_i}(x_i-v_i)} 
$$
This is the unique extension of $f$ to $Q$ that is moreover coordinatewise affine, i.e. its restriction to any segment parallel to any $e_i$, $i\leq n$ is an affine function. Crucially, we have $\lipnorm{\Lambda_Qf}=\lipnorm{f}$. See \cite[Section 3]{PerneckaSmith} or \cite[Lemma 3.2]{LancienPernecka} for the details.

Note that, for a given function $f\in\Lip(\ZZ_{\ell_1^n})$, the functions $\Lambda_Q(f\restrict_{V_Q})$ can be glued together nicely, i.e. if hypercubes $Q,Q'$ have non-empty intersection then $\Lambda_Q(f\restrict_{V_Q})$ and $\Lambda_{Q'}(f\restrict_{V_{Q'}})$ agree on $Q\cap Q'$: this follows immediately from coordinatewise affinity. Thus the extension $\Lambda_nf:\ell_1^n\to\RR$ of $f$, given by $\Lambda_nf(x)=\Lambda_Q(f\restrict_{V_Q})(x)$ where $x\in Q$, is well-defined. Moreover $\lipnorm{\Lambda_nf}=\lipnorm{f}$, as is easily seen by partitioning linear segments $[x,y]$ in $\ell_1^n$ into pieces contained entirely within some hypercube $Q$. Finally, the operator $\Lambda_n:\Lip(\ZZ_{\ell_1^n})\to\Lip(\ell_1^n)$ is clearly linear and pointwise continuous. Thus $\ZZ_{\ell_1^n}$ is weak$^*$ $1$-Lipschitz extendable in $\ell_1^n$. This is already mentioned explicitly in \cite[Corollary 3.8]{CCD}.

Now consider $f\in\Lip_0(\ZZ_{\ell_1})$ and the sequence of functions $f_n:=\Lambda_n(f\restrict_{\ZZ_{\ell_1^n}})\in\Lip_0(\ZZ_{\ell_1^n})$. It is clear that $f_n\restrict_{\ell_1^m}=f_m$ for $n>m$, again by coordinatewise affinity. Thus we may unambiguously define an extension $Ef:\bigcup_n\ell_1^n\to\RR$ of $f$ by $Ef(x)=f_n(x)$ where $x\in\ell_1^n$, and $\lipnorm{Ef}=\lipnorm{f}$. Since $\bigcup_n\ell_1^n$ is dense in $\ell_1$, this function extends uniquely to an extension $Ef\in\Lip_0(\ell_1)$ of $f$ that also satisfies $\lipnorm{Ef}=\lipnorm{f}$ (see \cite[Proposition 1.6]{Weaver2}). A standard argument shows that the map $f\mapsto Ef$ is pointwise continuous. This establishes part (c).

For parts (d) and (e), we use a similar argument as for part (c), building the extension cubewise for the intersection with $\ell_1^n$ for any $n$, and then extending from the intersection with $\bigcup_n\ell_1^n$, which is dense, to the whole set. 
\end{proof}

With this knowledge, the results from Section \ref{sec:homogeneous} allow us to improve on \cite[Proposition 3.7]{CCD}, which states that the countable $\ell_\infty$-sums of the following spaces are isomorphic.

\begin{theorem}\label{th:grid ell1}
$\Lip_0(\ZZ_{\ell_1})$ is isomorphic to $\Lip_0(\ell_1)$.
\end{theorem}

\begin{proof}
Let $S_n=Q_{4^n}\cap \ZZ_{\ell_1^n}$. This is a bounded set that is weak$^*$ extendable in $\ZZ_{\ell_1}$ with constant $1$: indeed, $S_n$ is extendable in either $Q_{4^n}\cap\ZZ_{\ell_1}$ or $\ZZ_{\ell_1^n}$ by the existence of retractions, and those are extendable in $\ell_1$, hence in its subset $\ZZ_{\ell_1}$, by Lemma \ref{lm:lemma_extension_Z_ell1}. Since $\ZZ_{\ell_1}$ is homogeneous, Proposition \ref{pr:homogeneous unbounded sequence} implies that $(\bigoplus_{n=1}^\infty\Lip_0(S_n))_\infty\compl\Lip_0(\ZZ_{\ell_1})$, and we also have $\Lip_0(\ZZ_{\ell_1})\compl\Lip_0(\ell_1)$ by Lemma \ref{lm:lemma_extension_Z_ell1}. Now notice that $2^{-n}S_n$ is the set of all points $x\in\ell_1^n$ such that each coordinate is a multiple of $2^{-n}$ and bounded by $\pm 2^n$, and that the sets $(2^{-n}S_n)$ form a nested sequence whose union is dense in $\ell_1$. Thus
$$
\Lip_0(\ell_1) \compl \pare{\bigoplus_{n=1}^\infty\Lip_0(2^{-n}S_n)}_\infty \equiv \pare{\bigoplus_{n=1}^\infty\Lip_0(S_n)}_\infty
$$
by Lemma \ref{lm:lip density}. Since $\Lip_0(\ell_1)$ is isomorphic to its countable $\ell_\infty$-sum by Lemma \ref{lm:kaufmann}, we conclude $\Lip_0(\ell_1)\sim\Lip_0(\ZZ_{\ell_1})$ by Pe\l czy\'nski's method.
\end{proof}

Theorem \ref{th:grid ell1} also answers the question raised in \cite[p. 2717]{CCD} on whether $\Lip_0(\ZZ_{\ell_1})$ is isomorphic to its countable $\ell_\infty$-sum. But we can go further and prove the predual statement: $\lipfree{\ZZ_{\ell_1}}$ is isomorphic to its countable $\ell_1$-sum, yielding another case where Question \ref{q:l1 sum} has a positive answer. For the proof, we need to describe a new extension operator in $\ell_1$.

\begin{lemma}
\label{lm:lemma_extension_partition_ell1}
Let $\mathcal{P}$ be a collection of pairwise disjoint subsets of $\NN$. Then $\bigcup_{I\in\mathcal{P}}\ell_1(I)$ is weak$^*$ $2$-Lipschitz extendable in $\ell_1$.
\end{lemma}

\begin{proof}
First fix some subset $I\subset\NN$ and let $P_I:\ell_1\to\ell_1(I)$ be the standard projection. We will define a $2$-Lipschitz retraction $R_I:\ell_1\to\ell_1(I)$ with the property that $R_Ix=0$ for all $x\in\ell_1$ that do not belong to the cone
$$
C_I := \set{x\in\ell_1 \,:\, \norm{x-P_Ix}<\norm{P_Ix}} .
$$
To that end, we first define the function
$$
r_I(x) := \max\set{1-\frac{\norm{x-P_Ix}}{\norm{P_Ix}},0}
$$
for $x\in\ell_1$, with the understanding that $r_I(x)=0$ if $P_Ix=0$. Note that $r_I$ is continuous and that $r_I(x)>0$ precisely when $x\in C_I$. Now set
$$
R_Ix := r_I(x)P_Ix .
$$
Then $R_I$ is a continuous retraction onto $\ell_1(I)$, as $r_I(x)=1$ if $x=P_Ix\neq 0$. We claim that $R_I$ is $2$-Lipschitz. To see this, fix $x,y\in\ell_1$. Assume first that $x,y\in C_I$, i.e. $r_I(x),r_I(y)>0$. Then
\begin{align*}
\norm{R_Ix-R_Iy} &\leq r_I(x)\norm{P_Ix-P_Iy} + \norm{P_Iy}\abs{r_I(x)-r_I(y)} \\
&\leq \norm{x-y} + \norm{P_Iy}\abs{\frac{\norm{x-P_Ix}}{\norm{P_Ix}} - \frac{\norm{y-P_Iy}}{\norm{P_Iy}}} \\
&= \norm{x-y} + \abs{\norm{x-P_Ix}-\norm{y-P_Iy}-\frac{\norm{P_Ix}-\norm{P_Iy}}{\norm{P_Ix}}\norm{x-P_Ix}} \\
&\leq \norm{x-y} + \big|\norm{x-P_Ix}-\norm{y-P_Iy}\big| + \frac{\norm{x-P_Ix}}{\norm{P_Ix}} \big|\norm{P_Ix}-\norm{P_Iy}\big| \\
&\leq \norm{x-y} + \norm{(x-y)-P_I(x-y)} + \norm{P_I(x-y)} \\
&= 2\norm{x-y} .
\end{align*}
By continuity, the same estimate holds for $x,y\in\cl{C_I}$.
Now assume that $r_I(x)>0,r_I(y)\notin\cl{C_I}$. Let $z$ be the closest point to $y$ in the linear segment $[x,y]$ such that $z\in\cl{C_I}$. Then $R_Iz=0$ by continuity, and
$$
\norm{R_Ix-R_Iy} = \norm{R_Ix-R_Iz} \leq 2\norm{x-z} \leq 2\norm{x-y} .
$$
This establishes our claim.

Next, observe that the cones $C_I$, $C_J$ are disjoint if $I,J$ are disjoint subsets of $\NN$. Indeed, if $x\in C_I$ then $\norm{P_Ix}>\norm{x-P_Ix}=\norm{P_{\NN\setminus I}x}\geq\norm{P_Jx}$, and a symmetric argument shows $\norm{P_Jx}>\norm{P_Ix}$ for $x\in C_J$. Both conditions are clearly incompatible.

Now let $f\in\Lip_0(\bigcup_{I\in\mathcal{P}}\ell_1(I))$ and define a function $Ef:\ell_1\to\RR$ by
$$
Ef(x) = \begin{cases}
f\circ R_I(x) &\text{, if $x\in C_I$}\\
0 &\text{, if $x\notin\bigcup_{I\in\mathcal{P}} C_I$}
\end{cases} .
$$
This function is well-defined because the cones $C_I$ are pairwise disjoint. It is now easy to verify that $\lipnorm{Ef}\leq 2\lipnorm{f}$. To do so, let $x,y\in\ell_1$. If $x,y\in\ell_1(I)$ for some $I\in\mathcal{P}$, then $Ef$ agrees with $f\circ R_I$ on $x$ and $y$, thus $\abs{Ef(x)-Ef(y)}\leq 2\lipnorm{f}\norm{x-y}$. Otherwise we may assume that $Ef(x)\neq 0$, which implies $x\in C_I$ for some $I\in\mathcal{P}$, and thus $y\notin C_I$. Let $u$ be the closest point to $y$ in the linear segment $[x,y]$ such that $u\in\cl{C_I}$, then $R_Iu=0$ and $Ef(u)=0$ by continuity. If $Ef(y)=0$ then $\abs{Ef(x)-Ef(y)}=\abs{Ef(x)-Ef(u)}\leq 2\norm{x-u}\leq 2\norm{x-y}$. Otherwise $y\in C_J$ for some $J\in\mathcal{P}$, $J\cap I=\varnothing$. Let $v$ be the closest point to $x$ in $[x,y]$ such that $v\in\cl{C_J}$; again, $R_Jv=0$ and $Ef(v)=0$. Since $C_I$, $C_J$ are open and disjoint, $u\in [x,v]$. Thus
\begin{align*}
\abs{Ef(x)-Ef(y)} &\leq \abs{Ef(x)}+\abs{Ef(y)} = \abs{Ef(x)-Ef(u)} + \abs{Ef(y)-Ef(v)} \\
&= \abs{f(R_I(x))-f(R_I(u))} + \abs{f(R_J(v))-f(R_J(y))} \\
&\leq 2\lipnorm{f}\norm{x-u}+2\lipnorm{f}\norm{y-v} \leq 2\lipnorm{f}\norm{x-y} .
\end{align*}
This finishes the proof that $\lipnorm{Ef}\leq 2\lipnorm{f}$.
So $E:\Lip_0(\bigcup_{I\in\mathcal{P}}\ell_1(I))\to\Lip_0(\ell_1)$ is a pointwise continuous extension operator with $\norm{E}\leq 2$.
\end{proof}

\begin{theorem}\label{th:gridsum}
$\lipfree{\ZZ_{\ell_1}}$ is isomorphic to its countable $\ell_1$-sum.
\end{theorem}

\begin{proof}
Let $(I_n)_{n\in\NN}$ be a partition of $\NN$ into countably many infinite subsets. It is clear that each $\ell_1(I_n)$ is isometric to $\ell_1$. Moreover, the spaces $\ell_1(I_n)$ are in $\ell_1$-sum with each other. In particular, for any choice of subsets $S_n\subset\ell_1(I_n)$ containing $0$, the sets $(S_n)$ are in \emph{metric $\ell_1$-sum}, meaning that $d(x,y)=d(x,0)+d(y,0)$ for $x,y$ belonging to different sets $S_n$. Thus, $\Lip_0(\bigcup_nS_n)$ can be identified isometrically with $\pare{\bigoplus_n\Lip_0(S_n)}_\infty$ via the natural gluing operation, and $\lipfree{\bigcup_nS_n}\equiv\pare{\bigoplus_n\lipfree{S_n}}_1$ (see e.g. \cite[Propositions 2.8 and 3.9]{Weaver2}). Applying this to $S_n=\ZZ_{\ell_1(I_n)}$, we get
$$
\mathcal{F}\pare{\bigcup_{n=1}^\infty\ZZ_{\ell_1(I_n)}} \equiv \pare{\bigoplus_{n=1}^\infty \lipfree{\ZZ_{\ell_1(I_n)}}}_1 \equiv \pare{\bigoplus_{\NN} \lipfree{\ZZ_{\ell_1}}}_1 .
$$
Now, Lemma \ref{lm:lemma_extension_Z_ell1} shows that each $\ZZ_{\ell_1(I_n)}$ is weak$^*$ $1$-Lipschitz extendable to $\ell_1(I_n)$. Thus, by the Lipschitz space identification mentioned above, $\bigcup_n\ZZ_{\ell_1(I_n)}$ is weak$^*$ $1$-Lipschitz extendable to $\bigcup_n\ell_1(I_n)$ simply by extending from each $\ZZ_{\ell_1(I_n)}$ to $\ell_1(I_n)$ individually. Additionally, Lemma \ref{lm:lemma_extension_partition_ell1} shows that $\bigcup_n\ell_1(I_n)$ is weak$^*$ Lipschitz extendable to $\ell_1$, hence also to its subset $\ZZ_{\ell_1}$. Therefore
$$
\pare{\bigoplus_{\NN} \lipfree{\ZZ_{\ell_1}}}_1 \equiv \mathcal{F}\pare{\bigcup_{n=1}^\infty\ZZ_{\ell_1(I_n)}} \compl \mathcal{F}\pare{\bigcup_{n=1}^\infty\ell_1(I_n)} \compl \lipfree{\ZZ_{\ell_1}} .
$$
Pe\l czy\'nski's method now completes the argument.
\end{proof}

\subsection{Compact Lipschitz retracts in Banach spaces}
\label{sec:compacts}

The question whether every separable Lipschitz-free space $\lipfree{M}$ is isomorphic to the Lipschitz-free space $\lipfree{K}$ over some compact metric space $K$ was recently answered in the negative by Basset \cite{Basset}. All counterexamples $M$ provided in \cite{Basset} are complete countable spaces. Thus, the restricted question whether the same can hold for all Lipschitz-free spaces $\lipfree{X}$ over separable Banach spaces $X$ is still open. One can even ask whether it would be possible to moreover choose $K\subset X$. This question was raised by Garc\'ia-Lirola and Proch\'azka in \cite{GP}, where they gave an affirmative answer when $X$ is Pe\l czy\'nski's universal basis space.

The next theorem is a step towards a possible positive solution to that problem. Recall that a \emph{finite-dimensional (Schauder) decomposition}, FDD for short, of $X$ is a sequence $(X_n)$ of non-trivial, finite-dimensional subspaces of $X$ such that every $x\in X$ has a unique representation $x=\sum_nx_n$ with $x_n\in X_n$. Schauder bases correspond to the particular case where $\dim(X_n)=1$ for all $n$.

\begin{theorem}\label{th:compactfinal}
Let $X$ be a Banach space admitting an FDD. Then there is a compact, convex, linearly dense subset $K\subset X$ that is a Lipschitz retract of $X$ and such that $\Lip_0(X)$ is isomorphic to $\Lip_0(K)$.
\end{theorem}

\begin{proof}
Let $(X_n)$ be an FDD of $X$, and let $P_n$ be the partial sum projection onto $E_n=X_1\oplus\ldots\oplus X_n$. Then $(E_n)$ is an increasing sequence of finite-dimensional subspaces of $X$ whose union is dense in $X$. It is proved in \cite{HM23} that there exists a compact, convex, linearly dense set $K\subset X$ that is moreover a Lipschitz retract of $X$. The construction in \cite{HM23} has one extra property that will be required in our proof: for every $n$, the intersection $K\cap E_n$ has non-empty interior in $E_n$. We may also assume that $0\in K$ for simplicity.

We will construct a sequence $(S_n)$ satisfying the hypotheses of Theorem \ref{th:lip extendable infinite union} (with $x_0=0$ and $X$ as the ambient space) such that $S_n$ is an $E_n$-ball contained in $K\cap E_n$. Let us first see how this will be enough to prove the theorem. By Theorem \ref{th:lip extendable infinite union} we deduce that $\bigcup_n S_n$ is weak$^*$ Lipschitz extendable to $M$, hence also to its subset $K$. Then we have
$$
\pare{\bigoplus_{n=1}^\infty \lipfree{B_{E_n}}}_1 \equiv \pare{\bigoplus_{n=1}^\infty \lipfree{S_n}}_1 \compl \mathcal{F}\pare{\bigcup_{n=1}^\infty S_n} \compl \lipfree{K} \compl \lipfree{X} .
$$
Indeed, the first relation holds because $S_n$ is a translated, rescaled copy of $B_{E_n}$; the second is due to Lemma \ref{lm:metric l1 sum}; and the third and fourth ones hold because of weak$^*$ Lipschitz extendability. On the other hand, since $\bigcup_n B_{E_n}$ is dense in $B_X$, Lemma \ref{lm:lip density} (or, more directly, \cite[Theorem 3.1]{CCD}) yields $\Lip_0(B_X)\compl (\bigoplus_n\Lip_0(B_{E_n}))_\infty$. Thus $\Lip_0(B_X)\compl\Lip_0(K)\compl\Lip_0(X)$, and we conclude that $\Lip_0(K)\sim\Lip_0(X)$ by Lemma \ref{lm:kaufmann} and Pe\l czy\'nski's method.

Now we proceed with the construction. For every $n$, $K\cap E_n$ has non-empty interior and thus, by convexity, contains arbitrarily small closed $E_n$-balls arbitrarily close to $0$. Thus we may iteratively choose balls $S_n = x_n+2^{-(N_n+1)}B_{E_n}\subset K\cap E_n$, where $\norm{x_n}=2^{-N_n}$ and $(N_n)$ is an increasing sequence in $\NN$ such that $N_n\geq N_{n-1}+2$. Let us check that the sets $S_n$ obtained in this way satisfy the conditions of Theorem \ref{th:lip extendable infinite union}:

(i) By the general properties of FDDs, $C:=\sup_n \norm{P_n}<\infty$. Thus each $E_n$ is a $C$-Lipschitz retract of $X$, and $S_n$ is a $2$-Lipschitz retract of $E_n$, proving that condition (i) is satisfied with constant $2C$.

(ii) We have $\diam(S_n)=2^{-N_n}$ and $d(S_n,0)=2^{-(N_n+1)}$, thus (ii) is satisfied with $D=\frac{1}{2}$.

(iii) Suppose that $x\in S_n$ and $y\in S_m$ with $n<m$. Then
$$
d(x,y) \geq d(x,0)-d(y,0) \geq 2^{-(N_n+1)} - 3\cdot 2^{-(N_m+1)} \geq 2^{-(N_n+1)} - 3\cdot 2^{-(N_n+3)} = 2^{-(N_n+3)}
$$
while
$$
d(x,0)+d(y,0) \leq 3\cdot 2^{-(N_n+1)} + 3\cdot 2^{-(N_m+1)} \leq 2\cdot 3\cdot 2^{-(N_n+1)} = 24\cdot 2^{-(N_n+3)}
$$
thus (iii) is satisfied with $\lambda=24$.

This finishes the construction, and the proof.
\end{proof}

\begin{remark}
If we only want to show that $\Lip_0(X)\sim\Lip_0(K)$ for some compact $K\subset X$, without requiring it to be convex or a Lipschitz retract, then we may use a simplified construction that does not depend on the results in \cite{HM23}. For instance, pick points $x_n\in E_n$ with $\norm{x_n}=2^{-2n}$ and consider the balls $S_n=x_n+2^{-(2n+1)}B_{E_n}$. Then $K=\cl{\bigcup_n S_n}=\bigcup_n S_n\cup\set{0}$ is compact (but not convex, nor a Lipschitz retract of $X$). The verification that $\bigcup_n S_n$ is weak$^*$ Lipschitz extendable and that $\Lip_0(X)\sim\Lip_0(K)$ stays essentially the same. We still require the existence of a Schauder decomposition in order to verify condition (i) in Theorem \ref{th:lip extendable infinite union}, and it has to be an FDD so that $K$ is compact.
\end{remark}

\subsection{Bundles of parallel subspaces}
\label{sec:bundles}

Our last result is independent of Section \ref{extensions}.
It is asked in \cite[Question 5]{Aliaga25} whether $\lipfree{\RR\times\ZZ}$ is isomorphic to $\lipfree{\RR^2}$. While we cannot answer that question, we can at least express the former in terms of standard Lipschitz-free spaces as $\lipfree{\RR\times\ZZ}\sim\lipfree{\RR}\oplus\lipfree{\ZZ^2}$. The same argument works, more generally, for higher dimensions.

\begin{proposition} \label{prop:bundles}
For $n,m\in\NN$, $\lipfree{\RR^n\times\ZZ^m}$ is isomorphic to $\lipfree{\RR^n}\oplus\lipfree{\ZZ^{n+m}}$.
\end{proposition}

Proposition \ref{prop:bundles} also holds trivially if $n=0$. The case $m=0$, that is $\lipfree{\RR^n}\oplus\lipfree{\ZZ^n}\sim\lipfree{\RR^n}$, follows easily from Pe\l czy\'nski's method as $\lipfree{\ZZ^n}\compl\lipfree{\RR^n}$ and $\lipfree{\RR^n}\sim(\bigoplus_\NN\lipfree{\RR^n})_1$. For the general case, we will need the following auxiliary lemma by Gartland \cite[Lemma 6.2]{Gartland25}.

\begin{lemma}[Gartland]\label{lm:gartland}
Let $M$ be a metric space with finite Nagata dimension. Then there exists $r\in (0,\infty)$ such that $\lipfree{M}\compl\lipfree{S}\oplus\pare{\bigoplus_{x\in S}\lipfree{B(x,r)}}_1$ for every maximal $1$-separated subset $S\subset M$.
\end{lemma}

\begin{proof}[Proof of Proposition \ref{prop:bundles}]
First, let $B$ a ball in $\RR^n$ that does not intersect $\ZZ^n$, and let $B'=B\times\set{(0,\stackrel{m}{\ldots},0)}\subset\RR^{n+m}$. Then we have
$$
\lipfree{\RR^n}\oplus\lipfree{\ZZ^{n+m}} \sim \lipfree{B'}\oplus\lipfree{\ZZ^{n+m}} \sim \lipfree{B'\cup\ZZ^{n+m}}\oplus\ell_1^2 \sim \lipfree{B'\cup\ZZ^{n+m}}
$$
by Lemmas \ref{lm:kaufmann}, \ref{lm:metric l1 sum} and \ref{lm:cdw}, respectively; we may apply Lemma \ref{lm:metric l1 sum} as $B'$ and $\ZZ^{n+m}$ are easily seen to be well-separated because $B'$ is bounded and separated from $\ZZ^{n+m}$ by a positive distance. Since the set $B'\cup\ZZ^{n+m}$ is a subset of $\RR^n\times\ZZ^m$ in the doubling metric space $\RR^{n+m}$, we deduce
$$\lipfree{\RR^n}\oplus\lipfree{\ZZ^{n+m}}\compl\lipfree{\RR^n\times\ZZ^m} .
$$

For the converse, equip $\RR^{n+m}$ with the $\ell_\infty$-norm metric. Then $\ZZ^{n+m}$ is a maximal $1$-separated subset of $\RR^n\times\ZZ^m$, so by Lemma \ref{lm:gartland} there exists $r\in (0,\infty)$ such that
$$
\lipfree{\RR^n\times\ZZ^m} \compl \lipfree{\ZZ^{n+m}}\oplus\pare{\bigoplus_{x\in\ZZ^{n+m}}\lipfree{B(x,r)\cap (\RR^n\times\ZZ^m)}}_1 .
$$
Note that all sets $B(x,r)\cap(\RR^n\times\ZZ^m)$, $x\in\ZZ^{n+m}$ are isometric and we may just write
$$
\lipfree{\RR^n\times\ZZ^m} \compl \lipfree{\ZZ^{n+m}}\oplus\pare{\bigoplus_\NN\lipfree{rB_{\RR^{n+m}}\cap (\RR^n\times\ZZ^m)}}_1 .
$$
Since $r<\infty$, the set $rB_{\RR^{n+m}}\cap (\RR^n\times\ZZ^m)$ is a union of finitely many (say $k$) isometric, parallel subsets of $\RR^n$, separated by a positive distance (hence well-separated), each of which has non-empty interior. Thus
$$
\lipfree{rB_{\RR^{n+m}}\cap (\RR^n\times\ZZ^m)} \sim  \lipfree{\RR^n}\oplus\ldots\oplus\lipfree{\RR^n}\oplus\ell_1^k \sim \lipfree{\RR^n}\oplus\ell_1^k \sim \lipfree{\RR^n}
$$
by Lemmas \ref{lm:metric l1 sum}, \ref{lm:kaufmann} and \ref{lm:cdw}. It follows that $\lipfree{\RR^n\times\ZZ^m} \compl \lipfree{\ZZ^{n+m}}\oplus\lipfree{\RR^n}$, using Lemma \ref{lm:kaufmann} again.

Finally, both $\lipfree{\RR^n}$ and $\lipfree{\ZZ^{n+m}}$ are isomorphic to their countable $\ell_1$-sums (e.g. by Lemma \ref{lm:kaufmann} and Theorem \ref{th:net l1 sum}), so we may apply Pe\l czy\'nski's method to conclude.
\end{proof}

\section*{Acknowledgments}

This research was carried out during visits of the first author to the Universidad P\'ublica de Navarra and of the second author to the Universitat Polit\`ecnica de Val\`encia in 2025. Both authors would like to thank the respective hosting institutions for their excellent conditions. The first author would also like to thank Gilles Lancien for some insights into the topic of grids in $\ell_1$. We also thank the anonymous referee for spotting a subtle issue in the proof of Theorem \ref{th:lip extendable infinite union}.

R. J. Aliaga was partially supported by Grant PID2021-122126NB-C33 and R. Medina has been supported by Grant PID2021-122126NB-C31, both funded by MICIU/AEI/10.13039/501100011033 and by ERDF/EU.


\end{document}